\documentclass[letterpaper, 10 pt, conference]{ieeeconf}
\usepackage{amsmath}
\usepackage{float}
\usepackage{mathtools}
\mathtoolsset{showonlyrefs}
\usepackage{amsthm}
\usepackage{cite}
\usepackage{amssymb}
\usepackage{graphicx}
\usepackage{lmodern}
\usepackage{cite}
\usepackage{xcolor}
\usepackage{mathrsfs}

\usepackage[utf8]{inputenc}
\usepackage[normalem]{ulem}
\makeatletter
\usepackage{comment}
\usepackage{lastpage} % Required to determine the last page for the footer
\usepackage{extramarks} % Required for headers and footers
\usepackage[bb=dsserif]{mathalpha} % MathBB 1
\usepackage{epstopdf} % epstopdf
\usepackage{amsmath}
\usepackage[linesnumbered,ruled,vlined]{algorithm2e}
\usepackage{bbm}
\usepackage{enumerate}

\SetKwComment{Comment}{/* }{ */}
\usepackage{tikz}
\usepackage{pgfplots}
\usepgfplotslibrary{groupplots,dateplot}
\usetikzlibrary{patterns,shapes.arrows, fadings}
\usetikzlibrary{automata} % Import library for drawing automata
\usetikzlibrary{shapes.geometric}
\usetikzlibrary{calc,shadows}
\usetikzlibrary{positioning} % ...positioning nodes
\usetikzlibrary{arrows} % ...customizing arrows
\tikzset{node distance=2.5cm, % Minimum distance between two nodes. Change if necessary.
every state/.style={ % Sets the properties for each state
semithick,
fill=gray!10},
initial text={}, % No label on start arrow
double distance=2pt, % Adjust appearance of accept states
every edge/.style={ % Sets the properties for each transition
draw,
->,>=stealth', % Makes edges directed with bold arrowheads
auto,
semithick}}
\pgfplotsset{compat=newest}
\newlength\figH
\newlength\figW
\theoremstyle{definition}
\newtheorem{definition}{\protect\definitionname}
\newtheorem{assumption}{\protect\assumptionname}
\newtheorem{lemma}{\protect\lemmaname}

\newtheorem{remark}{Remark}

\theoremstyle{plain}
\newtheorem{theorem}{\protect\theoremname}
\newtheorem{problem}{Problem}

\DeclareMathOperator*{\argmax}{arg\,max}
\DeclareMathOperator{\adj}{Adj}
\DeclareMathOperator{\tr}{Tr}

\newcommand{\pushright}[1]{\ifmeasuring@#1\else\omit\hfill$\displaystyle#1$\fi\ignorespaces}

\makeatother

\providecommand{\assumptionname}{Assumption}
\providecommand{\definitionname}{Definition}
\providecommand{\lemmaname}{Lemma}
\providecommand{\theoremname}{Theorem}
\providecommand{\propositionname}{Proposition}
\newcommand{\norm}[1]{\left\lVert#1\right\rVert}

\newcommand{\Expectation}[1]{\mathbb{E}\left[#1\right]}

\newcommand{\trace}[1]{\tr\left(#1\right)}
\newcommand{\Adjacent}[2]{\adj_k(#1, #2)=1}

\newcommand{\simplex}{\phi}

\newcommand{\SH}{\mathcal{S}_H}

\newcommand{\seta}{\mathscr{A}}

\newcommand{\action}{\alpha}

\newboolean{show_comments}
\setboolean{show_comments}{false} %Set to true/false to toggle author comments
\ifthenelse{\boolean{show_comments}}{
    \newcommand{\Alex}[1]{\textcolor{blue}{Alex:~#1}} %Author 1 comments
    \newcommand{\mh}[1]{\textcolor{red}{MH:~#1}} %Author 2 comments
    \newcommand{\AuthorC}[1]{\textcolor{cc_color}{CC:~#1}} %Author 3 comments
} {
    \newcommand{\Alex}[1]{}
    \newcommand{\mh}[1]{}
    \newcommand{\AuthorC}[1]{}
}
\IEEEoverridecommandlockouts  
\title{\LARGE \bf Guaranteed Feasibility in Differentially Private \\ Linearly Constrained Convex Optimization}
\author{
Alexander Benvenuti$^{1}$, Brendan Bialy$^{2}$, Miriam Dennis$^{2}$, Matthew Hale$^{1}$
\thanks{
$^{1}$School of  Electrical and Computer Engineering, Georgia Institute of Technology, Atlanta, GA USA.
Emails: \texttt{\{abenvenuti3,matthale\}@gatech.edu}.
}
\thanks{$^{2}$Munitions Directorate, Air Force Research Laboratory, Eglin AFB, FL USA.
Emails: \texttt{\{brendan.bialy,miriam.dennis.1\}@us.af.mil}.}
\thanks{
This work was partially supported by AFRL under grant FA8651-23-F-A008, 
NSF under CAREER grant 1943275, 
ONR under grant N00014-21-1-2502, AFOSR under grant FA9550-19-1-0169, and the NSF Graduate
Research Fellowship under grant DGE-2039655. Any opinions, findings and conclusions or recommendations expressed in this material are those of the authors and do not necessarily reflect the views of sponsoring agencies.
}
}

\begin{document}
% \onecolumn
% \aistatstitle{Guaranteed Feasibility in Differentially Private \\ Linearly Constrained Convex Optimization}
% \aistatsauthor{ Alexander Benvenuti \And Brendan Bialy \And Miriam Dennis \And  Matthew Hale }

% \aistatsaddress{ Georiga Institute of Technology \And Air Force Research Lab \And  Air Force Research Lab \And Georiga Institute of Technology}
\maketitle

\begin{abstract}
  Convex programming with linear constraints plays an important role in the operation of a number of everyday systems. 
  However, absent any additional protections, revealing
  or acting on the solutions to such problems may reveal information about their constraints, which can be sensitive. Therefore, in this paper, we introduce a method for solving convex programs while
  keeping linear constraints private. 
  First, we prove that this method is differentially private and always generates a feasible optimization problem (i.e., one whose solution exists). 
  Then we show that the solution to the privatized problem also satisfies the original, non-private constraints. 
    Next, we bound the expected loss in performance from privacy, which is measured by comparing
    the cost with privacy to that without privacy. 
    Simulation results apply this framework to constrained
    policy synthesis in a Markov decision process, and they 
    show that a typical privacy implementation induces only an approximately~$9\%$ loss in solution quality.
    %Numerical simulations highlight the effectiveness of this technique to provide good performance under strong privacy.
    %Simulation results show that performance losses are approximately~$9\%$ under a strong privacy implementation.
    % \mhmargin{Put a quantitative finding here.
    % It's fine to say ``Simulation results show
    % that performance losses are as low
    % as X\% under a typical privacy implementation.''
    % }
\end{abstract}

% \input{1-Intro}

% \input{2-Prelims_1}

% \input{3-Private Constraints}

% \input{4-Accuracy_1}

% %\input{5-LP}

% \input{6-Policy Synthesis}
\section{Introduction}
Convex optimization problems with linear constraints appear in many applications, such as power grids~\cite{stott1979review}, transportation systems~\cite{hoffman1963simple}, and resource allocation problems~\cite{markowitz1953portfolio}. The constraints 
in such problems can be sensitive, e.g., the load flow in a power grid, the time to travel between locations, and the costs of resources, which may reveal information about
individuals and/or trade secrets. 
%, \mh{such as...}. 
% \mhmargin{Preceding sentence: are those quantities in the constraints themselves? Make sure that they are.
% I don't want to accidentally say something that's true for databases but not true for what we're talking about. 
% }
The solutions to these problems are necessary for these systems to operate; however, simply computing and using these solutions may reveal the sensitive constraints used to generate them. 

Interest has therefore arisen in solving these  problems while both (i) preserving the privacy of constraints, and (ii) ensuring that all constraints are still satisfied at a solution.
%, but the state-of-the-art method for doing this allows for the constraints to be violated~\cite{hsu2014privately}, which is unacceptable in many cases, especially safety-critical systems. 
%\mh{Be very careful here. You've set this up beautifully, but we want to make sure that that it's impossible for a reader to misinterpret us. As written, there was a chance that someone reading this might say ``these people don't know what they're doing -- the Munoz paper ensures feasibility!'' I commented out the sentence that referred to the Hsu/Roth paper, which lets us get right to Munoz.}
%\Alex{Makes sense!}
In this paper we address an open question posed in~\cite{munoz2021private}, namely, protecting the privacy of constraints, specifically the constraint coefficient matrix, while maintaining feasibility with respect to the original, non-private constraints.
%\mh{What do you think of the following? It's weird to have math in the first paragraph, but it might work here. Don't hesitate to disagree!}
%\Alex{I like it, I was hesitant to have math in the intro, let alone the first paragraph, but if you think it works here I'm all for it, I think it's very clear and concise.}
For linear constraints~$Ax \leq b$,
the work in~\cite{munoz2021private} privatized~$b$
while ensuring constraint satisfaction, and~\cite{munoz2021private} identified the privatization of~$A$ with guaranteed constraint
satisfaction as an open problem. That is the problem we solve. 

To provide privacy to these constraints, we use differential privacy. Differential privacy is a statistical notion of privacy originally developed to protect entries in databases~\cite{dwork2006calibrating}. 
We use it in this work partly because of its immunity to post-processing~\cite{dwork2014algorithmic}, namely that arbitrary computations on private data do not weaken differential privacy. Therefore, we first privatize each constraint in the constraint coefficient matrix, then solve the resulting optimization problem, which is simply a form of post-processing
private data. Thus, the solution preserves the privacy of the constraints, as do any downstream computations that use it. 

Some common privacy mechanisms, e.g., the Gaussian and Laplace mechanisms~\cite{dwork2014algorithmic}, add noise with unbounded support. 
Here, such mechanisms can perturb constraints by arbitrarily far amounts, which can cause a solution
not to exist. 
Therefore, we use the truncated Laplace mechanism~\cite{munoz2021private, geng2020tight}, which allows us to privatize constraints such that they only become tighter. We use this property to ensure that 
(i) a solution always exists for a private problem, and (ii) 
the solution to the private problem satisfies the constraints of the original, non-private problem.
Then, we
bound the change in optimal cost due to privacy, 
which directly relates privacy to performance.   
%we expect changes in the optimal cost. Therefore, we bound the change in cost on average given a set of privacy parameters. These bounds allow users to rigorously trade off privacy and performance. 
%
% To illustrate our approach, 
% %We apply this formulation to a class of problems of particular interest,
% we apply it to policy synthesis of constrained Markov decision processes. We empirically show that even under strong privacy, our solution performs equally as well as one using the unbounded Laplace mechanism, despite the fact that we ensure constraint satisfaction while the other approach cannot. 
%
To summarize, our contributions are: 
\begin{itemize}
    \item We develop a differential privacy mechanism for the coefficient matrix in linear constraints %in convex optimization problems 
    (Theorem~\ref{thm:privacy}).
    \item We prove that privatized problems have solutions that satisfy the constraints of the original, non-private problem (Theorem~\ref{lem:feasible}).
    \item We bound the expected change in optimal cost due to
    privatizing constraints (Theorem~\ref{thm:lp_accuracy_e}).
    \item We empirically validate the performance of this method on constrained Markov decision processes (Section~\ref{sec:cmdp}).
\end{itemize}

\subsection{Related Work}
There is a large literature on differential privacy in optimization, specifically looking at privacy for objective functions~\cite{wang2016differentially, huang2015differentially, han2016differentially, nozari2016differentially, dobbe2018customized, lv2020differentially}. We differ because we privatize
constraints. 
Privacy for linear programming, a special case of convex optimization with linear constraints, was addressed in~\cite{hsu2014privately,cummings2015privacy,munoz2021private}. While~\cite{hsu2014privately, cummings2015privacy} both privatize constraints, 
they allow for constraint violation, which may be unacceptable, e.g., if constraints encode safety. 
As noted above, for constraints~$Ax \leq b$,~\cite{munoz2021private} privatizes~$b$,
while we privatize~$A$. 
%Additionally, the approach in~\cite{hsu2014privately} requires the use of the optimal objective function value, but that work does not address how to preserve the privacy of such solutions. 

%As noted above, the work in~\cite{munoz2021private} is the closest to ours, and for constraints of the form~$Ax \leq b$ it privatizes the~$b$ vector, which can depend on sensitive data. That work also named privacy for~$A$ as an open problem, and that is the problem that we solve.

\subsection{Notation}\label{subsec:notation}
For $N\in\mathbb{N}$, we use $[N] := \{1, 2, \ldots, N\}$. We use $\phi (B)$ to be the set of probability distributions over a finite set $B$, and $|\cdot|$ denotes the cardinality of a set. 
%We use~$\log$ to denote the natural logarithm and $\pi$ as a policy for a Markov decision process. 
$\trace{M}$ denotes the trace of a square matrix~$M$.
%, and~$\setmeasure(\mathcal{P})$ to denote the measure of a set~$\mathcal{P}$. 
%\mh{Do we mean ``Lebesgue measure''?}
We use~$\Expectation{X}$ to denote the expectation of a random variable~$X$.

\section{Background and Problem Formulation}\label{sec:background} 
This section gives background and problem
statements. 
\subsection{Convex Optimization with Linear Constraints}\label{subsec:LP}
We consider optimization problems of the form
    \begin{equation}
        \begin{aligned}
        &\begin{aligned}
            \underset{x}{\operatorname{maximize}} &\quad g(x)
        \end{aligned}
            \\
            &\begin{aligned}
                \text{ s.t. }Ax&\leq b\\
                x&\geq 0,
            \end{aligned}
        \end{aligned}   
        \tag{P}\label{opt:primal}
    \end{equation} 
where~$g:\mathbb{R}^n\to\mathbb{R}$ is $L$-Lipschitz and concave,~$A\in\mathbb{R}^{m\times n}$ is the ``constraint coefficient matrix'', and~$b\in\mathbb{R}^m$ is the ``constraint vector''. We use~$\seta$
to denote the set of all possible~$A$ matrices. 
\begin{assumption}\label{ass:bounded}
    The set~$\seta$ is bounded and the bounds are publicly available.
\end{assumption} 

Assumption~\ref{ass:bounded} is quite mild since the entries of~$A$ may represent physical quantities that do not exceed certain bounds, e.g., with voltages
in a power grid. 
% With Assumption~\ref{ass:bounded} we then have the mapping
% \begin{equation}\label{eq:db_map}
%     \Amap:\seta\to\mathbb{R}^{m\times n}
% \end{equation}
% which maps the set of databases to the constraint coefficient matrix. 
%That is, certain entries of the constraint coefficient matrix are generated from user data in some way, and thus, absent any protections, the solution to Problem~\eqref{opt:primal} can reveal information about~$D$. We use~$A(D)$ to denote that the constraint matrix depends on the underlying  database, and with this notation the constraints may be divided into constraints that depend on sensitive data and constraints that do not.
% Supposing that there are~$m_1$ constraints that
% depend on sensitive data and~$m_2$ constraints
% that do not, we have 
%     \begin{equation}
%         \begin{aligned}
%         &\begin{aligned}
%             \underset{x}{\operatorname{maximize}} &\quad g(x)
%         \end{aligned}
%             \\
%             &\begin{aligned}
%                 \text{ subject to }
%                 A(D)_px&\leq b_p\\
%                 A_{np}x&\leq b_{np}\\
%                 x&\geq 0,
%             \end{aligned}
%         \end{aligned}   
%         \tag{S}\label{opt:private}
%     \end{equation} 
%One can consider
%the case when only some constraints are sensitive, and this case would be handled by dividing the constraints into those that require privacy and those that do not. 
We consider the case in which all constraints require privacy, though this approach can be applied as-is to any subset of constraints. 
% where~$A(D)_p\in\mathbb{R}^{m_1\times n}$ are coefficients of the~$m_1$ constraints that depend on sensitive data and~$A_{np}\in\mathbb{R}^{m_2\times n}$ are the coefficients of the~$m_2$ constraints that do not depend on sensitive data,~$b_{p}\in\mathbb{R}^{m_1}$ and~$b_{np}\in\mathbb{R}^{m_2}$ are their associated constant vectors that do not depend on sensitive data, and~$m_1+m_2 = m$. 
% We will consider the case where all entries of~$A$ depend on the database~$D$ and may be written as~$A(D)x\leq b$. 
% \mh{To simplify, why not just say ``One can consider
% the case when only some constraints depend on sensitive data, and this would be handled by dividing the constraints into those that require privacy and those that do not. We consider the case in which all constraints require privacy, though this approach can be applied as-is to any subset of the constraints in a problem.''
% Otherwise we kind of take a digression with new notation, etc. that isn't really useful to us. 
% }
\begin{assumption}\label{ass:measure}
    The feasible region~$\{x \in \mathbb{R}^n : Ax\leq b\}$ has non-empty interior for all~$A\in\seta$.
    %, that is,~$\setmeasure\left(\{x:A(D)x\leq b\}\right)>0$. 
\end{assumption}
% \mh{Is this for all~$D \in \mathcal{D}$? 
% Or something else? 
% Also, do we need to use~$\mu$ here?
% Saying ``non-empty interior'' is perfectly
% clear without needing to use~$\mu$. 
% }

\begin{remark}\label{rem:equality}
    If Assumption~\ref{ass:measure} fails, then any perturbation to the constraints can cause infeasibility, making such constraints fundamentally incompatible with privacy.
    %\mh{This feels like an assumption showing up much later.
    %Can we state this earlier?
    %}
\end{remark}

It has been observed in the 
literature~\cite{hsu2014privately,cummings2015privacy,munoz2021private}
that the public release of the solution to Problem~\eqref{opt:primal} 
% \mhmargin{I'm not seeing the Problem~P label.
% Maybe call the original problem Problem~0.}
%\mh{What is Problem~S?}
may reveal the constraint coefficient matrix used to generate it, and thus we will apply differential privacy to protect~$A$.

% For robust linear programming, let~$\mathscr{U}_i$ be the set of possible values for the row~$A_i$. The uncertainty set is then
% \begin{equation}
%     \mathscr{U} = \mathscr{U}_1\times \mathscr{U}_2\times\cdots\times \mathscr{U}_m.
% \end{equation}
% This then gives us the linear program which we will then solve:
%     \begin{equation}
%         \begin{aligned}
%         &\begin{aligned}
%             \underset{x\in\mathbb{R}_+^n}{\operatorname{minimize}} &\quad g(x)
%         \end{aligned}
%             \\
%             &\begin{aligned}
%                 \text{ subject to }A(D)x&\leq b \quad \forall A\in\mathscr{U}
%             \end{aligned}
%         \end{aligned}   
%         \tag{RC}\label{opt:robust}
%     \end{equation} 
% Solving problem~\eqref{opt:robust} has been investigated in~\cite{ben1999robust}, where the problem is transformed into a conic quadratic program (CQP), which can be solved in polynomial time. We plan to leverage this approach in attempt to attain accurate bounds on the solution.
\subsection{Differential Privacy}
We will apply differential privacy directly to the constraint matrix~$A$, and this approach is known as ``input perturbation" in the literature. Overall, the goal of differential privacy is to make ``similar" pieces of data appear approximately indistinguishable. The notion of ``similar" is defined by an adjacency relation. Many adjacency relations exist, and we present the one used in the remainder of the paper; see~\cite{dwork2014algorithmic} for additional background.
\begin{definition}[Adjacency]\label{def:adj2}
    For a constant~$k > 0$, two vectors~$v, w\in\mathbb{R}^n$ are said to be adjacent if there exists an index~$j \in [n]$
such that (i)~$v_i = w_i$ for all~$i \in [n] \backslash \{j\}$
and (ii)
    %\begin{equation}\label{eq:adj}
        $\norm{v-w}_1\leq k.$
    %\end{equation}
    %if the following conditions hold:
    %\begin{enumerate}[(i)]
        %\item there exist indicies~$i$ and~$j$ such that~$v_i = w_i$ for all~$i\in[n]\setminus\{j\}$ and~$v_j\neq w_j$
        %\mh{We don't actually need~$v_j \neq w_j$.
        %The reason is that we want the adjacency relation
        %to be symmetric, i.e., everything should be adjacent
        %to itself, so then we get an actual mathematical
        %``relation'' and not something else.
        %}
    %\end{enumerate}
    If two vectors~$v$ and~$w$ are adjacent, we write~$\Adjacent{v}{w}$; otherwise 
    we write~$\adj_k(v, w)=0$. 
    %\loz
\end{definition}
In words, two vectors are adjacent if they differ in up to one entry and that difference is bounded by~$k$. 
% \mh{The ``up to one entry'' part is missing from the above definition.
% I think we'd want to say that ``There exists an index~$j \in [n]$
% such that~$v_i = w_i$ for all~$i \in [n] \backslash \{j\}$
% and...[the 1-norm bound holds].'' What do you think?
% }

To make adjacent pieces of data appear approximately indistinguishable,
we implement differential privacy, which is done using a
randomized map called a ``mechanism". 
In its general form, differential privacy
protects a sensitive piece of data~$x$ by
randomizing some function of it, say~$f(x)$.
In this work, the sensitive data we consider
is the matrix~$A$, and we privatize
the output of the identity map acting
on~$A$, which privatizes~$A$ itself.
This is known as ``input perturbation'',
and next we define differential privacy 
for this approach.

%For sensitive data~$x\in\mathbb{R}^m$, a mechanism~$\mathscr{M}$ approximates~$x$ according to the following definition.
% \mh{This part here says ``$f : \mathbb{R}^m \to \mathbb{R}$'',
% but the definition below has a mechanism $\mathbb{R}^n \to \mathbb{R}$. 
% Make sure that these agree.}

\begin{definition}[Differential Privacy; \cite{dwork2014algorithmic}]\label{def:dp}
    Fix a probability space $(\Omega, \mathcal{F}, \mathbb{P})$. Let~$k > 0$,~$\delta\in[0, \frac{1}{2})$, and $\epsilon>0$ be given. A mechanism $\mathscr{M}:\mathbb{R}^{m}\times\Omega\to\mathbb{R}^m$ is ($\epsilon, \delta)$-differentially private if for all $v, w\in\mathbb{R}^{m}$ that are adjacent in the sense of Definition \ref{def:adj2},  we have
        %\begin{equation}
            $\mathbb{P}[\mathscr{M}(v)\in T] \leq e^\epsilon \mathbb{P}[\mathscr{M}(w)\in T]+\delta$
             for all Borel measurable sets  $T\subseteq\mathbb{R}^{m}$. 
        %\end{equation}
        %\hfill$\lozenge$
\end{definition}
%We note that because adjacency is a symmetric definition, i.e.,~$\adj_k(v, w) = \adj_k(w, v)$, we can exchange the places of~$v$ and~$w$ in Definition~\ref{def:dp} and it must still hold. 

The strength of privacy is set by~$\epsilon$ and~$\delta$, and smaller values of both imply stronger privacy. The value of~$\epsilon$ quantifies the leakage of sensitive information, and typical values for it are~$0.1$ to~$10$~\cite{hsu2014differential}.
The value of~$\delta$ can be interpreted as the probability
that more information is leaked than~$\epsilon$ allows, and typical
values for~$\delta$ range 
from~$0$ to~$0.05$~\cite{hawkins2023node}. 

A feature
of differential privacy is that it is
immune to post-processing, which we formalize
next. 

% \mh{Move this part [until my next comment below] to Section III. }
 
% \mh{End the part that's getting moved. 
% This is indeed background information, so
% normally it would go here, but the mechanism
% itself is an implementation choice that we're making,
% which is more like a contribution than background.
% And parallel composition goes with the mechanism.
% This also has the benefit of making the problem
% statement show up earlier. What do you think?
% }
% \mh{
% If seems like we partition the matrix~$A(D)$ into~$A_1(D), \ldots, A_m(D)$,
% but this partitions the output of~$d$, rather than
% the input to~$d$, which is the database~$D$. 
% In other words, parallel composition requires
% us to partition~$D$ into~$D_1, \ldots, D_N$, and then we would
% want each row of~$A$ to depend only on a single~$D_i$.
% Otherwise, if each row of~$A$ can depend on the entirety
% of~$D$, then we can't use parallel composition here
% since each row is a query of the whole database. 
% Can you clarify
% this point?
% }
%One useful quality of differential privacy is its immunity to post-processing. That is, arbitrary computations on differentially private data remain differentially private, which is formalized as follows. 

\begin{lemma}[Immunity to Post-Processing; \cite{dwork2014algorithmic}]\label{lem:arbitrary}
    Let~$\mathscr{M}:\mathbb{R}^n\times \Omega \to \mathbb{R}^m$ be an~$(\epsilon, \delta)$-differentially private mechanism. Let~$h:\mathbb{R}^m\to\mathbb{R}^p$ be an arbitrary mapping. Then the composition~$h\circ\mathscr{M}:\mathbb{R}^n\to\mathbb{R}^p$ is~$(\epsilon, \delta)$-differentially private.%\hfill $\triangle$
\end{lemma}

In the context of convex programming with linear constraints, Lemma~\ref{lem:arbitrary} implies that we can privatize the constraint coefficient matrix~$A$, and the solution to the privatized optimization 
problem preserves the privacy of~$A$ by virtue of being post-processing. 

\subsection{Problem Statements}
Consider Problem~\eqref{opt:primal} . 
%\mh{What is Problem~S?}
Computing~$x^*$ depends on the sensitive constraint coefficient matrix~$A$, and simply computing and using~$x^*$ can reveal information about~$A$. Therefore, we aim to develop a framework for solving problems in the form of Problem~\eqref{opt:primal} that preserves the privacy of~$A$ while still satisfying the constraints in Problem~\eqref{opt:primal} . 
This will be done by solving the following problems.
\begin{problem}\label{prob:private}
    Develop a privacy mechanism to privatize the constraint coefficient matrix.
\end{problem}

\begin{problem}\label{prob:feasible}
    Prove that the privacy mechanism produces constraints such that the solution to the privately generated optimization problem also satisfies the constraints of the original, non-private problem.
\end{problem}

\begin{problem}\label{prob:accurate}
    Bound the average change in the cost between the non-private optimal solution and the privatized optimal solution.
\end{problem}

\section{Private Linear Constraints} \label{sec:privacy}
In this section, we solve Problems~\ref{prob:private} and~\ref{prob:feasible}. Specifically, we (i) detail our approach to implementing privacy for the~$A$ matrix, (ii) prove that is it~$(\epsilon, \delta)$-differential private, and (iii) show that solutions computed with private constraints also satisfy
the corresponding non-private constraints. 
% \mh{It seems like we switch between using~$A$ and~$A(D)$. Maybe
% that's fine, but check on it.}\Alex{We want~$A(D)$ when talking specifically about our sensitive constraint matrix. So we make the switch from saying~$A$ to~$A(D)$ right after Problem P.}
We note that entries of~$A$ that equal zero may represent that there is physically no relationship between a decision variable and a constraint. For example, in a Markov decision process that models a traffic system,
a zero transition probability may indicate that one street
does not connect to another, which is publicly known. 
Thus, we consider the number of zero entries to be public information, and 
%\mh{For example...[give a super short example of that here. 1 sentence is plenty]}
 we will only privatize the non-zero entries. 

\subsection{Implementing Differential Privacy}
 For a given~$A$ matrix, 
 we use~$a^{0}_{i}$ to denote the vector of non-zero entries in row $i$, and we use~$a^{0}_{i,j}$ to denote the~$j^{th}$ entry in that vector. 
% \mh{It seems like~$a_i$ is row~$i$, and then~$a^{0}_i$ is the
% vector of non-zero entries of~$a_i$. And then~$a_{i,j}$
% is the~$j^{th}$ entry of~$a_i$ and~$a^0_{i,j}$ is
% the~$j^{th}$ entry of~$a^0_i$. If so, say all of this explicitly
% here before proceeding.}\Alex{Added!}
To implement privacy we will 
compute
\begin{equation} \label{eq:tildeaij}
\tilde{a}^{0}_{i,j} = a^{0}_{i,j} + s_i +z_{i, j},
\end{equation}
where~$z_{i,j}$ is the noise added to each entry to enforce privacy; we add~$s_i$ to enforce that all the coefficient values become larger, thus tightening the constraints. 
% \mh{
% This is where to put the material from above.
% Blend it all together to make it sound good. 
% }

We will add bounded noise to ensure that~$z_{i, j}$ only tightens the constraints when privatizing~$A$; if the constraints are only tightened,
then privacy can only shrink the feasible region, and thus
satisfaction of the privatized constraints implies satisfaction
of the original, non-private constraints. 
We do this with the truncated Laplace mechanism. 

\begin{lemma}[Truncated Laplace Mechanism;~\cite{munoz2021private, geng2020tight}]\label{lem:t_lap_mech}
    Let privacy parameters~$\epsilon>0$ and~$\delta\in(0, \frac{1}{2}]$ be given, and fix the adjacency relation from Definition~\ref{def:adj2}.
    %For a function~$f : \mathbb{R}^m \to \mathbb{R}^m$,
    The Truncated Laplace Mechanism takes sensitive data~$x\in\mathbb{R}^m$ as input and outputs private data~$z \in \mathbb{R}^m$, where~$z_i \in \mathcal{S}$ and ${z_i\sim\mathcal{L}_T(\sigma, \mathcal{S})}$ for all~$i\in[m]$. Here,~$\mathcal{L}_T(\sigma, \mathcal{S})$ is the truncated Laplace distribution with density
    %\begin{equation}
        $f(z_i) = \frac{1}{Z_i}\exp\left(-\frac{1}{\sigma}|z_i|\right),$
    %\end{equation}
    where~$\mathcal{S} := [-s, s]$ and the values of~$s$ and~$-s$ are bounds on the support of the private outputs. We define~$Z_i=\mathbb{P}(z_i\leq |s|)$, and~$\sigma$ is the scale parameter of the distribution. The truncated Laplace mechanism is~$(\epsilon, \delta)$-differentially private if~$\sigma\geq \frac{k}{\epsilon}$ and
    %\begin{equation}
        $s = \left[\frac{k}{\epsilon}\log\left(\frac{m(e^\epsilon-1)}{\delta}+1\right) \right].$
    %\end{equation}
    %\hfill $\triangle$
\end{lemma}

We apply this mechanism to each row of a constraint coefficient matrix~$A$, which provides row-wise privacy.
This approach in fact provides privacy to the
entire~$A$ matrix due to the following. 
%This implementation then provides privacy to the
%entire~$A$ matrix due to the following lemma. 
%Next we state a lemma on how the application of privacy row-wise affects the privacy of the entire matrix~$A$.

\begin{lemma}[Parallel Composition;~\cite{ponomareva2023dp}]
    \label{lem:parallel-comp}
    Consider a database~$D$ partitioned into disjoint subsets~$D_1, D_2, \ldots, D_N$,
    and suppose that there are privacy mechanisms~$\mathscr{M}_1, \mathscr{M}_2, \ldots, \mathscr{M}_N$,
    where~$\mathscr{M}_i$ is~$(\epsilon, \delta)$-differentially private for all~$i \in [N]$. 
    Then the release of the queries~$\mathscr{M}_1(D_1), \mathscr{M}_2(D_2), \ldots, \mathscr{M}_N(D_N)$
    provides~$D$ with~$(\epsilon,\delta)$-differential privacy.     
\end{lemma}

We consider~$A$ and partition it
into its rows, then privatize each row
individually. 
Lemma~\ref{lem:parallel-comp} ensures that
doing so provides~$(\epsilon,\delta)$-differential
privacy to the~$A$ matrix as a whole. 
% \mh{Two thoughts. First, why are we adding~$s_i$ here?
% Second, the truncated Laplace mechanism doesn't look
% like it adds noise. I get what you're saying, but is there a cleaner
% connection to make here?
% Our truncated Laplace definition says ``draw~$z_i$ from this distribution'', but that doesn't scream ``add noise''.
% }\Alex{We add~$s_i$ to enforce that the constraints are being tightened. I added some text to try to make this all more clear.} 

Along with privacy, 
we must also enforce feasibility. 
In order to guarantee that the privately obtained solution~$\tilde{x}^*$ is feasible with respect to the non-private problem, it is clear 
that the two problems must have
at least one feasible point in common. 
We state this formally in the following assumption.
% \mh{I think that the assumption is stronger than what we say right here.
% Is it more accurate to say that ``There must exist at least
% one point that satisfies the constraints produced by
% every realization of the database~$D$''?
% }\Alex{Added this after the assumption.}

\begin{assumption}[Perturbed Feasibility]\label{ass:feast}
    The set~$S=\bigcap_{A\subseteq\seta}\left\{x: Ax\leq b\right\}$ is not empty.
\end{assumption}

In words, Assumption~\ref{ass:feast} says that there must exist at least
one point that satisfies the constraints produced by
every realization of the constraint coefficient matrix~$A$. 
With Assumption~\ref{ass:feast}, we post-process~$\tilde{a}_{i,j}$ from~\eqref{eq:tildeaij} 
to obtain
\begin{equation} \label{eq:abarpost}
\bar{a}^{0}_{i,j} = \min\big\{\tilde{a}^0_{i,j}, \sup_{\seta} a^{0}_{i,j} \big\}, 
\end{equation}
and we 
do so for each~$(i, j)$ such that~$a_{i,j}$ is non-zero.
The output of these computations is the private
constraint coefficient matrix, denoted~$\tilde{A}$. 

\begin{remark}
    Taking the minimum in~\eqref{eq:abarpost} ensures that each entry in~$\tilde{A}$ appears in some matrix in~$\seta$. 
    % \mh{Can you be more specific? 
    % Are we saying that ``Each entry in~$\tilde{A}$ appears
    % in some matrix in~$\mathscr{A}$''?
    % }
    The supremum is finite since~$\seta$ is bounded and does not depend on sensitive information according to Assumption~\ref{ass:bounded}. 
\end{remark}
With this, we solve the optimization problem
    \begin{equation}
        \begin{aligned}
        &\begin{aligned}
            \underset{x}{\operatorname{maximize}} &\quad g(x)
        \end{aligned}
            \\
            &\begin{aligned}
                \text{ s.t. }\tilde{A}x&\leq b\\
                x&\geq 0.
            \end{aligned}
        \end{aligned}   
        \tag{DP-P}\label{opt:DP}
    \end{equation} 

Algorithm~\ref{algo:solve} provides a unified
overview of our approach. 

    \begin{algorithm}
    \caption{Privately Solving Convex Problems with Linear Constraints}
    \SetKwFor{ForAll}{for all}{do}{end}
    \label{algo:solve}
    \SetAlgoLined
    \textbf{Inputs}: Problem~\eqref{opt:primal},
    $\epsilon$, $\delta$, $k$;\\
    \textbf{Outputs}: Privacy-preserving solution $\tilde{x}^*$;\\
    %\hline
    \vspace{1mm}
    Set $\sigma = \frac{k}{\epsilon}$;\\
    \ForAll{$i\in[m]$}
        {Count the non-zero entries in row $i$, namely $n^0_i$;\\
        Compute the support for the truncated Laplace mechanism, i.e., $s_i =\left[\frac{k}{\epsilon}\log\left(\frac{n^0_i(e^\epsilon-1)}{\delta}+1\right) \right]$;\\
        \ForAll{$j\in[n]$}
        {Generate $\tilde{a}^{0}_{i,j} = a^{0}_{i,j} + s_i +z_{i, j}$ using the truncated Laplace Mechanism in~\eqref{eq:tildeaij};\\
        Post-process using~\eqref{eq:abarpost}, i.e., compute $\bar{a}^{0}_{i,j} = \min\{\tilde{a}^0_{i,j}, \sup\limits_{\seta} a^{0}_{i,j}\}$;}}
    Form $\tilde{A}$ by replacing each non-zero entry~$a^0_{i,j}$ in~$A$ with $\bar{a}^{0}_{i,j}$;\\
    Solve Problem~\eqref{opt:DP} (via any algorithm) to find~$\tilde{x}^*$
\end{algorithm}

% \mhmargin{
% Are we missing a ``for'' loop over~$j$
% It seems like lines 7 and 8 should loop over~$j$.
% Thoughts?
% }

\subsection{Characterizing Privacy}
Next we prove that Algorithm~\ref{algo:solve} preserves the privacy of~$A$.

\begin{theorem}[Solution to Problem~\ref{prob:private}]\label{thm:privacy}
    Fix an adjacency parameter~$k > 0$, let privacy parameters~$\epsilon>0$ and~$\delta\in[0, \frac{1}{2})$ be given, and let Assumptions~\ref{ass:bounded}-\ref{ass:feast} hold. Then Algorithm~\ref{algo:solve}
    % \mh{Here we could replace~``Problem~\eqref{opt:DP}'' with ``Algorithm~1''}\Alex{How does that look?}
    keeps~$A$~$(\epsilon, \delta)$-differential private with respect to the adjacency relation in Definition~\ref{def:adj2}.
\end{theorem}
\begin{proof}
    See Appendix~\ref{pf:privacy}.
\end{proof}
Theorem~\ref{thm:privacy} allows us to privatize each constraint individually, and the resulting constraint coefficient matrix~$\tilde{A}$ will be~$(\epsilon, \delta)$-differentially private. 
\begin{figure*}[t]
\begin{equation}\label{eq:xi}
    \xi = \begin{cases}
        \sqrt{\sum_{j=1}^{m}\left(2m\left(\frac{k}{\epsilon}\right)^2 n^0_j+ \left(\frac{n^0_jk}{\epsilon}\log\left(\frac{n_j^0(e^{\epsilon}-1)}{\delta}+1\right)\right)^2\right)} &\text{if }a_{i,j}+2s_i< \sup_{\seta} a_{i,j} \text{ for all }i\text{ and }j\\
        \norm{A-\bar{A}}_F & \text{if there exists  } i\text{ and }j \text{ s.t. }a_{i,j}+2s_i\geq \sup_{\seta} a_{i,j}.
    \end{cases}
\end{equation}
\hrulefill
\end{figure*}Solving Problem~\eqref{opt:DP} then preserves the privacy of~$A$, and the resulting solution, namely~$\tilde{x}^*$, can be released and/or acted on without harming privacy.

\begin{theorem}[Solution to Problem~\ref{prob:feasible}]\label{lem:feasible}
    Fix an adjacency parameter~$k > 0$,
    let privacy parameters~$\epsilon>0$ and~$\delta\in[0, \frac{1}{2})$ be given, and let Assumptions~\ref{ass:bounded}-\ref{ass:feast} hold. Fix a constant vector~$b$. Then the optimization problem~\eqref{opt:DP} produces a solution that also satisfies the original, non-private constraints. 
\end{theorem}
\begin{proof}
   See Appendix~\ref{pf:feasible}.
\end{proof}

% \mh{By ``feasible'', do we mean that
% ``the problem has a non-empty feasible region''
% or ``the solution to the problem with privatized
% constraints also satisfies the original, non-private
% constraints''?
% }

Theorem~\ref{lem:feasible} guarantees that after privacy is implemented, the resulting optimization problem is feasible. Since Algorithm~\ref{algo:solve} only tightens the constraints, this implies that a privatized solution~$\tilde{x}^*$ always exists and is in the feasible set of the original, non-private optimization problem.

In the next section, we will analyze the accuracy of the solution to the privatized problem.

% \mh{Devil's advocate: the result with equality constraints
% is super interesting, but do we want to include it?
% We're saying ``Here's a case where we fail'', which might
% be low-hanging fruit for a lazy reviewer. I'm not necessarily
% opposed to including it just yet, but we need to tread carefully,
% especially with these AI conferences. They have short timelines
% for their reviews, and it might be easy for a reviewer to say 
% ``Well, this doesn't really work, as they show. And I really think
% equality constraints are just so common that they had better
% figure that case out before this is ready to submit.''
% It would also be easy for them to miss the point that we're making
% and say ``Well it only works sometimes! How do you even know
% when it will work?'' The answer to that is stated here, but, again,
% a lazy person flipping through and looking at the figures might
% miss that.
% What do you think?
% }\Alex{That makes sense, so are you suggesting we state in a remark or assumption up front that ``Hey, this is only possible when you have a non-empty interior to your feasible region.", and just leave it at that? I feel like we should address it somewhere to answer the easy question of ``What about equality constraints?" I added an assumption and moved the remark up to the convex optimization section so that we address this up front and in a way that is not calling out equality constraints by name.}
\section{Solution Accuracy}
%We expect a loss in objective with the private solution as it will lie in the interior of the non-private feasible set. Accordingly, we bound this loss next.
In this section, we solve Problem~\ref{prob:accurate}. To do so, we compute an upper bound on change in cost between the nominally generated solution and the privately generated solution. This bound depends on (i) the Lipschitz constant of the objective function, (ii) the largest feasible solution of the original, non-private constraint coefficient matrix, 
%\mh{Can you clarify~(ii)? Should ``giving'' be ``given''?}
(iii) the largest possible constraint coefficients allowable from~$\seta$, and (iv) the ``closeness" of the private and non-private feasible regions. 

For (i), we state this as an assumption below.

\begin{assumption}\label{ass:lips}
    The objective function~$g:\mathbb{R}^n\to\mathbb{R}$ is~$L$-Lipschitz with respect to the~$\ell_2$-norm on~$\mathbb{R}^n$.
\end{assumption}
\noindent For (ii), we define
\begin{equation}\label{eq:abar}
    \bar{A} = [\sup_{\seta} a_{i,j}]_{i\in[m], j\in[n]}
\end{equation}
as the matrix where each entry is its largest value in the set of all constraint coefficient matrices~$\seta$.
For (iii), we define 
\begin{equation}\label{eq:feas_set}
F(A) = \{x \in \mathbb{R}^n : Ax \leq b\}
\end{equation}
as the feasible region of the original, non-private constraints given a choice of~$A$. Then we define
\begin{equation}\label{eq:xbar}
    \bar{x} \in \argmax_{x \in F(A)} \|x\|_2, 
\end{equation}
which is any element of largest~$2$-norm in the feasible region generated by the original, non-private constraints.

For (iv), two ``close'' sets of linear inequalities have a bounded difference in their feasible regions. 

\begin{lemma}[Perturbation Bound; \cite{hoffman1952approximate}]\label{lem:perturbation}
    % \mh{
    % I don't think this lemma is really about Problem~P. Why not
    % just say ``Consider a set of linear inequalities
    % of the form~$Ax \leq b$, with~$A \in \mathbb{R}^{m \times n}$
    % and~$b \in \mathbb{R}^m$.''
    % }
    Given~$Ax \leq b$ with~$A \in \mathbb{R}^{m \times n}$
    and~$b \in \mathbb{R}^m$, consider~$F(A)$ from~\eqref{eq:feas_set}. 
    For a matrix~$\hat{A}\in\mathbb{R}^{m\times n}$ and vector~$\hat{b}\in\mathbb{R}^m$, let~$\hat{x}$ be such that
    %\begin{equation}
        $\hat{A}\hat{x}\leq \hat{b}.$
    %\end{equation}
    Then there exists an~$x\in F(A)$ such that
    \begin{equation}
        \norm{x-\hat{x}}_2\leq H_{2, 2}(A)\norm{\left[(A-\hat{A})\hat{x}-(b-\hat{b})\right]^+}_2,
    \end{equation}
    % where, given a norm~$\norm{\cdot}_p$, we use~$\norm{\cdot}_q$ to denote is its dual norm,
    where~$[\cdot]^+$ is the projection onto the non-negative orthant of~$\mathbb{R}^m$, 
    and~$H_{2, 2}(A)$ is the Hoffman constant of~$A$, i.e., 
    \begin{equation}
        H_{2, 2}(A) = \max_{J\in\mathcal{J}(A)} \frac{1}{\min\big\{\|A_J^T v\|_2: v\in\mathbb{R}^{J}_+, \norm{v}_2 = 1\big\} },
    \end{equation}
    where (i) $\mathcal{J}(A) = \{J\subseteq [m]: \seta_J(\mathbb{R}^n)+\mathbb{R}_+^{|J|} = \mathbb{R}^{|J|}\}$ is the set of all~$J$ such that the set-valued mapping~$x\to A_Jx+\mathbb{R}_+^{|J|}$ is surjective, (ii) $A_J$ is the matrix formed by deleting all rows of~$A$ whose indices are not in~$J$, and (iii) we define~$\seta_J(\mathbb{R}^n) = \{z\in\mathbb{R}^n : A_Jz\leq b_J\}$, where~$b_J$ is
    formed by deleting entries of~$b$ whose
    indices are not in~$J$. 
\end{lemma}

% \mh{In this section, we reference~$q$-norms
% a lot, but it seems like we can set~$q = 2$
% and~$p=2$ and then this is all a lot more
% concrete. What do you think?}

%The exact computation of Hoffman constants is NP-Hard, and~\cite{pena2018algorithm} provides an  algorithm to approximate them, which we use later. 
Next, we bound the expected change in cost.

\begin{theorem}[Solution to Problem~\ref{prob:accurate}]\label{thm:lp_accuracy_e}
    Fix an adjacency parameter~$k > 0$, 
    fix privacy parameters~$\epsilon>0$ and~$\delta\in[0, \frac{1}{2})$, and let Assumptions~\ref{ass:bounded}-\ref{ass:lips} hold. Let~$x^*$ be the solution to Problem~\eqref{opt:primal} and~$\tilde{x}^*$ be the solution to Problem~\eqref{opt:DP}. 
    Then
\begin{equation}
    \Expectation{|g(x^*)- g(\tilde{x}^*)|} \leq L\norm{\bar{x}}_2H_{2,2}(A)\xi,
\end{equation}
where~$\xi$ is from~\eqref{eq:xi},~$H_{2, 2}(A)$ is from Lemma~\ref{lem:perturbation},
$\bar{A}$ is from~\eqref{eq:abar}, 
and~$\bar{x}$ is from~\eqref{eq:xbar}.
% \begin{equation}
%     \xi = \begin{cases}
%         \sqrt{\sum_{j=1}^{m}\left(2m\left(\frac{k}{\epsilon}\right)^2 n^0_j+ \left(\frac{n^0_jk}{\epsilon}\log\left(\frac{n_j^0(e^{\epsilon}-1)}{\delta}+1\right)\right)^2\right)} &\text{if }a_{i,j}+2s_i< \sup_{A\in\seta} a_{i,j} \text{ for all }i\text{ and }j\\
%         \norm{A-\bar{A}}_F & \text{if there exists  } i\text{ and }j \text{ s.t. }a_{i,j}+2s_i\geq \sup_D a_{i,j}.
%     \end{cases}
% \end{equation}
\end{theorem}

% \subsection{Attempt 1}
% \begin{proof}
%     We start by applying the linearity of the expectation:
% \begin{equation}
%     \Expectation{g(x^*)- g(\tilde{x}^*)} = g(x^*) - \Expectation{g(\tilde{x}^*)}
% \end{equation}
% An upper bound on~$g(x^*) - \Expectation{g(\tilde{x}^*)}$ requires a lower bound on~$\Expectation{g(\tilde{x}^*)}$, which we obtain via Jensen's inequality:
% \begin{equation}
%     g(x^*) - \Expectation{g(\tilde{x}^*)} \leq g(x^*) - g(\Expectation{\tilde{x}^*}).
% \end{equation}
% We then use the Lipschitz properties of~$g$ to bound the above as
% \begin{equation}
%     g(x^*) - g(\Expectation{\tilde{x}^*})\leq L\norm{x^* - \Expectation{\tilde{x}^*}}.
% \end{equation}
% \Alex{bone-headed solution} Since~$\tilde{x}$ is element-wise lower bounded by the element-wise smallest feasible~$x\in\{x\in\mathbb{R}^n: Ax\leq b\}$, which we denote~$\underline{x}$, by the linearity of the expectation, then element-wise~$\Expectation{\tilde{x}^*}\geq \underline{x}$, and thus we obtain
% \begin{equation}
%     g(x^*) - g(\Expectation{\tilde{x}^*})\leq L\norm{x^* - \underline{x}}
% \end{equation}
% which completes the proof. \Alex{This bound is crazy loose and should just serve as a benchmark for when I have the next better thing.}
% \end{proof}
% \subsection{Attempt 2}
% Let's try this again trying to get something that actually depends on the privacy parameters.
\begin{proof}
See Appendix~\ref{pf:lp_accuracy_e}.
\end{proof}

Theorem~\ref{thm:lp_accuracy_e} 
%bounds the expected loss in performance as a function of the original constraint matrix, and feasible solution with the largest norm, and the privacy parameters. This theorem then 
allows users to anticipate performance loss before privacy is implemented and before a solution is computed, which enables the calibration of privacy based on its tradeoffs with performance.

\section{Application to Policy Synthesis}\label{sec:cmdp}
This section applies our developments to constrained Markov decision processes,
which we define next. 

%\subsection{Constrained Markov Decision Processes Background}\label{subsec:MDP}
%We first define a constrained Markov Decision Process.
\begin{definition}[Constrained Markov Decision Process;~\cite{altman2021constrained}] \label{def:cmdp}
A Constrained Markov Decision Process (CMDP) is the tuple~$\mathcal{M} = (\mathcal{S}, \mathcal{A}, r, \mathcal{T}, \mu, f, f_0)$, where $\mathcal{S}$ and $\mathcal{A}$ are the finite sets of states and actions, 
%\mh{Why ``local'' here? Why not just ``states and actions''?}
respectively, and~$|\mathcal{S}| = p$ and~$|\mathcal{A}| = q$. Then, $r:\mathcal{S}\times \mathcal{A}\rightarrow \mathbb{R}$ is the reward function, $\mathcal{T}:\mathcal{S}\times \mathcal{A}\rightarrow \simplex(\mathcal{S})$ is the transition probability function, $\mu\in\simplex(s)$ is a probability distribution over the initial states, $f_i:\mathcal{S}\times \mathcal{A}\to [0, f_{\text{max}, i}]$ for~$i\in[N]$ are immediate costs, and~$\Expectation{\sum_{t=0}^{\infty}\gamma^t f(s_t)}\leq f_0$,~$f_{0,i}\in\mathbb{R}$  are constraints. 
% \mh{Check the dimensions here. 
% If the output of~$f$ is in the interval~$[0, f_{max}]$
% and~$f_0 \in \mathbb{R}^n$, then it doesn't seem like
% we can have the expectation term be less
% than or equal to~$f_0$. 
% }
% \loz
\end{definition}

We let~$\mathcal{T}(s, \action, y)$ denote the probability of transitioning from state~$s$ to state~$y$ when taking action~$\action$.
We consider CMDPs in which the constraint function~$f$
is linear. Then the constraints
can be written as~$AX - f_0 \leq 0$,
where~$A \in \mathbb{R}^{pq\times N}$, where~$p$, $q$, and~$f_0$ are from
    Definition~\ref{def:cmdp} and~$X$ is the decision
    variable in policy synthesis,
    described next. 

Solving an MDP, i.e., computing the optimal policy, or list of commands to take in each state, can be done with the linear program~\cite{puterman2014markov}
%However, given that we have a CMDP, we use the linear programming approach as allows for the inclusion of constraints.
%\begin{lemma}[CMDP LP;~{\cite{puterman2014markov, altman2021constrained}}]\label{lem:puterman_probs}
%    The optimal value function~$v^*(s)$ for all~$s\in\mathcal{S}$ can be found by solving the linear program
%    \begin{equation}
%        \begin{aligned}
%        &\begin{aligned}
%            \underset{ v_{\pi}}{\operatorname{minimize}} &\quad \mu^Tv_{\pi}
%        \end{aligned}
%            \\
%            &\begin{aligned}
%                \text{ subject to }\quad &v_{\pi}(s) - \gamma\sum_{s'\in \mathcal{S}}\mathcal{T}(s,a,s')v_{\pi}(s')\\
%                &\geq r(s,a) \quad
%                \forall s\in \mathcal{S},a\in \mathcal{A}\\
%                &h(v_{\pi}(s))\leq h_0\quad \forall s\in\mathcal{S}.
%            \end{aligned}
%        \end{aligned}   
%        \tag{MDP-P}\label{opt:MDP-P}
%    \end{equation} 
%    where~$h:\mathcal{S}\to [0, h_{\text{max}}]$ are constraints on the MDP. 
%In particular we solve the linear program
    \begin{equation}
        \begin{aligned}
        &\begin{aligned}
            \underset{ x_{\pi}}{\operatorname{maximize}} \quad &\sum_{s\in S}\sum_{\action\in A}&r(s, \action)x_{\pi}(s, \action)
        \end{aligned}
            \\
            &\begin{aligned}
                &\text{ s.t. } \,\, x_{\pi}(s, \action) \geq 0 \quad f(x_{\pi}(s, \action)) \leq f_0\quad \forall s\in\mathcal{S}, \action\in\mathcal{A} \\
                &\sum_{\action'\in \mathcal{A}} \!\! x_{\pi}(s'\!,\action') \!-\! \gamma
               \!\! \sum_{s\in \mathcal{S}}\!\sum_{\action\in \mathcal{A}}\!\!x_{\pi}(s,\action)\mathcal{T}(s,\action,s') \!=\! \mu(s')~\forall s' \!\in\! \mathcal{S}. 
                %& \forall s\in\mathcal{S}, a\in\mathcal{A}.
            \end{aligned}
        \end{aligned}    
        \tag{MDP-D}\label{opt:MDP-D}
    \end{equation} 
    % \mh{Has the ``primal form'' been defined
    % before this? Does this just mean
    % it's the form from Assumption~5?
    % If so, this Lemma should include ``Let
    % Assumption~5 hold'' in its statement. 
    % }
%\end{lemma}
%In this work, we consider functions~$f$ that satisfy Assumption~\ref{ass:linear} in Lemma~\ref{lem:puterman_probs}.  
%It is often preferred to solve Problem~\eqref{opt:MDP-D}. 
%The variables~$x_{\pi}(s, a)$ are known as ``occupancy measures'', and
%and they can be interpreted as  the number of times a policy~$\pi$ takes action~$a$ in state~$s$ when a policy is executed over a long time horizon. Formally,
%\begin{align}\label{eq:occ_def}
%    x_{\pi}(s,a)&=\sum_{t=0}^{\infty}\gamma^{t-1}P(s_t=s,a_t=a)\\
%    &=\Expectation{\mathbb{I}\left\{s_t=s,a_t=a\right\}}. 
%\end{align}
%Once the dual problem has been solved for the optimal occupancy measures~$\left\{ x^*_{\pi}(s,a)\right\} _{s\in\mathcal{S},a\in\mathcal{A}},$
The optimal policy~$\pi^*$ can be calculated 
from the optimum~$\left\{ x^*_{\pi}(s,\action)\right\} _{s\in\mathcal{S},\action\in\mathcal{A}}$ via 
%\begin{equation}\label{eq:om2policy}
    $\pi^*(\action \mid s)=\frac{x^*_\pi(s,\action)}{\sum_{\action'\in\mathcal{A}}x^*_\pi(s,\action')}$.
%\end{equation}
Such a policy admits a value function~$v_{\pi}$, which is defined as
%\begin{equation}\label{eq:value_iteration}
    $v_{\pi}(s) =  \Expectation{\sum_{t=0}^{\infty} \gamma^t r(s_t, \pi(s_t))},$
%\end{equation} 
and is easily computable~\cite{puterman2014markov}.
Constraints that may be encoded by~$A$ include enforcing a probability of reaching a goal state and safety, i.e., setting a maximum amount of visits to a set of hazardous states, with hazard values assigned to each state. 

This type of safety extends the example in~\cite{chow2018lyapunov} by allowing states to have varying ``hazard" factors, which are the immediate costs for each state-action pair. Let~$\SH \subseteq \mathcal{S}$ be the set of hazardous states, and let~$f(s_t) = \beta_s\mathbb{I}\{s_t\in\SH\}$, 
where~$\mathbb{I}$ is an indicator function, 
which encodes that the agent incurs a penalty of~$\beta_s$ for occupying state~$s$. 
%\mh{very briefly say what this means}.
We then have the constraint that
%\begin{equation}
    $\Expectation{\sum_{t=0}^{\infty} \gamma^t \beta_{s_t} f(s_t)|s_0, \pi}\leq f_0$, or, equivalently,~$f(x_{\pi}(s, \action)) = \beta_{s_t}\gamma\mathbb{I}\{x_{\pi}(s, \action)\in\SH\}\leq f_0$. As a linear constraint, this takes the form $AX - f_0 \leq 0$,
    where~$A$ is a vector since it is a single constraint, and entry~$i$ of~$A$ is defined as
\begin{equation}
    a_{i} =\begin{cases}
        \beta_{s_t}\gamma & \text{if }x(s_t, \action)\in \SH\\
        0 & \textnormal{otherwise}
    \end{cases}.
\end{equation}
%We then must compute the sensitivity of this constraint with respect to the adjacency parameter in Definition~\ref{def:adj2}. We do so next.
% \begin{theorem}[Differentially Private Safety Constraint]\label{thm:dp_safety}
%     Given privacy parameters~$\epsilon>0$ and~$\delta\in[0, \frac{1}{2})$, adjacency parameter~$k>0$, a CMDP~$\mathcal{M} = (\mathcal{S}, \mathcal{A}, r, \mathcal{T}, \mu, f, f_0)$ and a hazardous region~$\SH = \{s_1, s_2,\ldots, s_{\alpha}\}$ with associated hazard values~$\beta_1,\beta_2,\ldots, \beta_{\alpha}$, the mapping from the constraint~$f(x_{\pi}(s, a)) = \beta_{s}\gamma\mathbb{I}\{x_{\pi}(s, a)\in\SH\}$ to a policy~$\tilde{\pi}^*$ is~$(\epsilon, \delta)$-differentially private with respect to the adjacency relation in Definition~\ref{def:adj2}.
% \end{theorem}
% \begin{proof}
%    See supplementary material Section~\ref{pf:dp_safety}.
% \end{proof}

These are the constraints that we will privatize. 
%Using Proposition~\ref{prop:contract}, we evaluate the policy generated by solving Problem~\ref{opt:MDP-DP} and 
We will also 
empirically evaluate the ``cost of privacy'' using the metric  in~\cite{gohari2020privacy, benvenuti2023differentially}, namely 
the 
%absolute difference in the value function at an initial state~$|v_{\pi^*}(s_0)- v_{\tilde{\pi}^*}(s_0)|$. Specifically, we express the cost of privacy as a 
percent decrease in the value function, equal to~$\big((v_{\pi^*}(s_0)- v_{\tilde{\pi}^*}(s_0))/v_{\pi^*}(s_0)\big) \times 100\%$.

We apply Algorithm~\ref{algo:solve} to a CMDP representing the system in Figure~\ref{fig:safety_grid}, and we privatize the hazard values of each hazardous state. We set~$\beta_i = 1$ for all~$i\in\SH$,
and we define~$\seta$ so that~$\sup_{\seta} a_{i,j} = 3$.
%In an adversarial setting, avoiding specific states reveals what information is known about the adversary, which may be sensitive.
% \mh{I think we should change this.
% First, ``in an adversarial setting'' might make
% more sense as ``If a third party can observe
% an agent's states and actions, then they may
% infer the constraints that the agent is subject to.
% As a result, that third party may be able to infer
% what an agent knows (and does not know) about
% its environment, which can be sensitive.''
% }
%
%
%
The cost of privacy for~$\epsilon\in[0.01, 5]$ averaged over~$100$ samples is shown in Figure~\ref{fig:MDP_example}. With strong privacy, i.e.,~$\epsilon = 2$, we see a~$18.35\%$ reduction in performance, while with more typical privacy levels, i.e.,~$\epsilon= 3$, we see only a~$9.45\%$ reduction in performance, indicating that our method can simultaneously provide both desirable
privacy protections and desirable levels of performance.

\section{Conclusion}\label{sec:conclusion}
We presented a differentially private method for solving convex optimization problems with linear constraints while ensuring that constraints
are never violated. 
%We have shown that the average error induced by privacy is bounded. 
Future work will address the privatization of 
nonlinear and stochastic constraints while ensuring
their satisfaction as well. 
% \subsection*{Acknowledgments}
% \Alex{Comment this out for initial anonymous submission} This work was partially supported by NSF under CAREER grant \#1943275, AFRL under grant \#FA8651-23-F-A008, ONR under \#N00014-21-1-2502, AFOSR under grant \#FA9550-19-1-0169, and the National Science Foundation Graduate
% Research Fellowship under Grant \#DGE-2039655. Any opinions, findings and conclusions or recommendations expressed in this material are those of the authors and do not necessarily reflect the views of sponsoring agencies.

\begin{figure}
    \centering
    \includegraphics[scale=0.8]{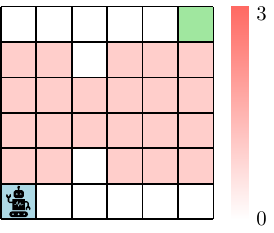}
    \caption{Grid in which the agent starts at the blue state,
    its goal is the green state, and hazardous states are red. 
    }
    \label{fig:safety_grid}
\end{figure}
    \begin{figure}
    \centering
        \input{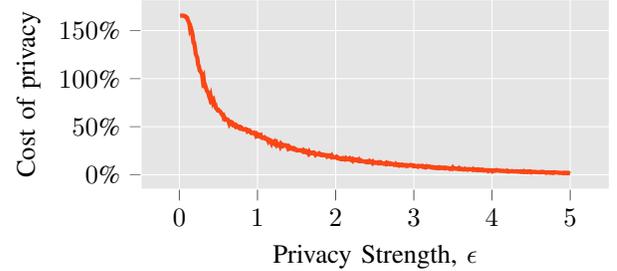}
        \caption{Cost of privacy for
        the example shown in Figure~\ref{fig:safety_grid} for
        a range of values of~$\epsilon$ implemented with
        Algorithm~\ref{algo:solve}.}
        \label{fig:MDP_example}
    \end{figure}

\bibliographystyle{IEEEtran}
\bibliography{main}
%\input{9-Checklist}
%\input{8-Appendix}
% \onecolumn
% \aistatstitle{Guaranteed Feasibility in Differentially Private Linearly Constrained Convex Optimization:
% Supplementary Materials}
% \thispagestyle{empty}
% \setcounter{section}{0}
\appendix
    \subsection{Proof of Theorem~\ref{thm:privacy}}\label{pf:privacy}
    From Lemma~\ref{lem:t_lap_mech}, computing~$\tilde{a}^0_{i,j} = a^0_{i,j} +s_i +z_{i,j}$ is~$(\epsilon, \delta)$-differentially private if~$z_{i,j}\sim\mathcal{L}_T(\frac{\Delta}{\epsilon}, \mathcal{S}_i)$, where $\mathcal{S}_i = [-s_i, s_i]$ and
    %\begin{equation}
        $s_i = \left[\frac{k}{\epsilon}\log\left(\frac{n_i^0(e^{\epsilon}-1)}{\delta}+1\right) \right].$
    %\end{equation}
    Additionally, computing~$\bar{a}^{0}_{i,j} = \min\{\tilde{a}^0_{i,j}, \sup_{\seta} a^{0}_{i,j}\}$ is merely post-processing of differentially private data, and thus from Lemma~\ref{lem:arbitrary}, maintains~$(\epsilon, \delta)$-differential privacy. Repeating this for each~$i\in[m]$, each vector of non-zero constraint coefficients is~$(\epsilon, \delta)$-differentially private, and from Lemma~\ref{lem:parallel-comp}, replacing the non-zero elements in~$A$ with~$\bar{a}_{i,j}$ to form~$\tilde{A}$ is~$(\epsilon, \delta)$-differentially private.
    \mh{Be more specific here than ``the matrix compose of each vector...'' It seems obvious how to put these together, but spelling it out won't take much space and can prevent ambiguity.}\Alex{Made this more clear.}
    From Lemma~\ref{lem:arbitrary}, it then follows that~$\tilde{x}^*$ is~$(\epsilon, \delta)$-differentially private by virtue of being post-processing of the differentially private quantity~$\tilde{A}$.

    \subsection{Proof of Theorem~\ref{lem:feasible}}\label{pf:feasible}
     By definition, the constraint matrix~$\tilde{A}$ is component-wise less than or equal to the matrix~$\bar{A}$ in which~$\bar{A}_{i,j}
     = \sup_{\seta} a_{i,j}$ for all~$i \in [n]$ and~$j \in [m]$. 
    %\begin{equation}
    %    \bar{A} = \begin{bmatrix}
    %        \sup_{\seta} a_{1,1} & \sup_{\seta} a_{1,2} & \cdots & \sup_{\seta} a_{1,n}\\
    %        \sup_{\seta} a_{2,1}& \sup_{\seta} a_{2, 2} & \cdots & \sup_{\seta} a_{2,n}\\
    %        \vdots & \vdots & \vdots & \vdots\\
    %        \sup_{\seta} a_{m, 1}& \sup_{\seta} a_{m, 2} & \cdots & \sup_{\seta} a_{m,n}m\\
    %    \end{bmatrix}.
    %\end{equation}
    Since~$x\geq 0$ and the vector~$b$ is fixed, we have that the set~$\{x:\bar{A}x\leq b\}$ is contained
    % \mh{Do we really mean ``interior''? I think we just mean
    % ``is contained in''. They might share a boundary, which
    % preserves containment but violates the interior condition.}\Alex{Addressed this}
    in~$\{x:\tilde{A}x\leq b\}$ due to the fact that $a_{i,j}\leq \sup_{\seta} a_{i,j}$, and as a result, we know that $\{x:\tilde{A}x\leq b\}\supseteq \{x:\bar{A}x\leq b\}$. 
    % \mh{Here we have~$a_{i,j}$ and~$a_{i,j}$ meaning
    % very different things, but the symbol is the same.
    % Can we change one?
    % } \Alex{Going with just~$a_{i,j}$}
    % \mh{What is~$b$ here? Is it a single vector that is fixed
    % throughout all of the different sets that we consider?
    % If so, say that.}\Alex{Added to the theorem statement and proof that we are fixing the~$b$ vector}
    
    We will now show that the sets~$\bigcap_{A\in\seta}\{x:Ax\leq b\}$ and~$\{x:\bar{A}x\leq b\}$ are equal. 
    For any~$x$ in the first set, 
    we know that $a_i\cdot x\leq b_i$ for all~$A \in \seta$. 
    By definition of the supremum, it follows then that~$\sup_{\seta} (a_i\cdot x)\leq b_i$, and therefore~$\bar{A}x\leq b$. As a result, if~$x\in\bigcap_{A\in\seta}\{z:Az\leq b\}$, then~$x\in\{z:\bar{A}z\leq b\}$. 
        We now show that the reverse is true. If~$\bar{A}x\leq b$, then~$Ax\leq b$ for all~$A \in \seta$
     by definition of the supremum. 
     Thus, if~$x\in\{z:\bar{A}z\leq b\}$, then we also have~$x\in\bigcap_{A\in\seta}\{z:Az\leq b\}$. 
     %This completes the proof that~$\bigcap_{A\in\seta}\{x:Ax\leq b\} = \{x:\bar{A}x\leq b\}$.
     
    Since we have~$\{x:\tilde{A}x\leq b\}\supseteq \{x:\bar{A}x\leq b\}$ and~$\{x:\bar{A}x\leq b\} = \bigcap_{A\in\seta}\{x:Ax\leq b\}$, we know that~$\{x:\tilde{A}x\leq b\}\supseteq\bigcap_{A\in\seta}\{x:Ax\leq b\}$. From Assumption~\ref{ass:feast}, the set~$\bigcap_{A\in\seta}\{x:Ax\leq b\}$ is non-empty, and therefore the set~$\{x:\tilde{A}x\leq b\}$ is non-empty and thus is yields a feasible optimization problem.

\subsection{Proof of Theorem~\ref{thm:lp_accuracy_e}}\label{pf:lp_accuracy_e}
The Lipschitz property of~$g$ from Assumption~\ref{ass:lips} gives
\begin{equation}\label{eq:sub_norm_here}
    \Expectation{|g(x^*)- g(\tilde{x}^*)|} \leq L\Expectation{\norm{x^{*} - \tilde{x}^*}}.
\end{equation}
% Now we look to bound the expectation above. Note that the following inequalities are true of both solutions
% \begin{align}
%     Ax^*&\leq b\\
%     \tilde{A}\tilde{x}^*&\leq b.
% \end{align}
% Subtracting both of these yields
% \begin{equation}
%     Ax^* - \tilde{A}\tilde{x}^*\leq 0.
% \end{equation}
% Note that~$\tilde{A}$ can be written as~$A+Z$, where~$Z$ is a matrix of random variables whose distributions we will investigate later. We then write the above inequality as
% \begin{equation}
%     Ax^* - (A+Z)\tilde{x}^* \leq 0\implies A(x^* - \tilde{x}^*) \leq Z\tilde{x}^*.
% \end{equation}
% Noting that~$A$ need not be square, but that we may consider a full-row rank~$A$ matrix without loss of generality, we use the psuedo-inverse of~$A$ to obtain
% \begin{equation}
%     x^* - \tilde{x}^* \leq A^{\dagger}Z\tilde{x}^*.
% \end{equation}
% Taking the 2-norm and expectation of both sides, we obtain
% \begin{equation}
%     \Expectation{\norm{x^* - \tilde{x}^*}_2} \leq \Expectation{\norm{A^{\dagger}Z\tilde{x}^*}_2}.
% \end{equation}
Noting that~$b$ remains constant \mh{between the feasible regions [thing with A] and [thing with the other A]}\Alex{Added!} between the feasible regions~$\{x:Ax\leq b\}$ and~$\{x:\tilde{A}x\leq b\}$, we apply Lemma~\ref{lem:perturbation} and the linearity of the expectation to obtain to obtain
    % \begin{multline}\label{eq:bound_by_hoffman1}
    %     \norm{x^*- \tilde{x}^*}_2 \leq H_{2,2}(A)\norm{\left[(A-\tilde{A})\tilde{x}\right]^+}_2 \\\leq H_{2,2}(A)\norm{\left[A-\tilde{A}\right]^+}_{F}\norm{\tilde{x}}_{2} 
    % \end{multline}
    \begin{multline}\label{eq:bound_by_hoffman1}
        \Expectation{\norm{x^*- \tilde{x}^*}_2} \leq %\Expectation{H_{2,2}(A)\norm{\left[(A-\tilde{A})\tilde{x}^*\right]^+}_2} \\\leq 
        H_{2,2}(A)\Expectation{\norm{\left[A-\tilde{A}\right]^+}_{F}\norm{\tilde{x}^*}_{2}}. 
    \end{multline}
    %where we choose the~$2$-norm for the Hoffman constant. 
    %Note that since~$\tilde{x}\geq 0$ by definition, any negative entries in~$(A-\tilde{A})\tilde{x}$ are negative only in~$A-\tilde{A}$, thus the bound holds.
    % \mh{I don't follow what the preceding sentence is getting at.}\Alex{I was trying to justify why it's ok to pull the~$\tilde{x}^*$ out of the projection, but it may just be easier to just remove that sentence and just let it be.}
    Since~$\tilde{x}$ is in a subset of the feasible space in the non-private problem, the largest possible~$\tilde{x}$ is bounded by the largest feasible~$x\in F(A)$, where~$F(A)$ is from~\eqref{eq:feas_set}. We denote this as %~$\bar{x}$ and say
    %\begin{equation}
        $\bar{x} \in \argmax_{x \in F(A)} \|x\|_2.$
    %\end{equation}
    % where
    % \begin{equation}
    %     \mathcal{D} = \{x \in \mathbb{R}^n : Ax \leq b\}. 
    % \end{equation}
    Then in~\eqref{eq:bound_by_hoffman1}  we may write the bound
    % \begin{multline}\label{eq:bound_by_hoffman}
    %     H_{2,2}(A)\norm{\left[A-\tilde{A}\right]^+}_{F}\norm{\tilde{x}}_{2} \\ \leq H_{2,2}(A)\norm{\left[A-\tilde{A}\right]^+}_{F}\norm{\bar{x}}_{2} 
    % \end{multline}
    %\begin{equation}\label{eq:bound_by_hoffman}
        $\Expectation{\norm{\left[A-\tilde{A}\right]^+}_{F}\norm{\tilde{x}^*}_{2}} \leq \Expectation{\norm{\left[A-\tilde{A}\right]^+}_{F}}\norm{\bar{x}}_{2}$. 
    %\end{equation}
    Next we bound~$\norm{[A-\tilde{A}]^+}_F$. First, we define~$Z=\tilde{A}-A$ and, using the non-expansive property of the projection onto the non-negative orthant, we obtain
%Using the Cauchy-Schwarz inequality, we then bound the expression above as
\begin{equation}\label{eq:sub_z_here}
     \Expectation{\norm{\left[A-\tilde{A}\right]^+}_{F}}\norm{\bar{x}}_{2} \leq \Expectation{\norm{Z}}\norm{\bar{x}}_{2}.
\end{equation}
% Since~$\tilde{x}$ is in a subset of the feasible space in the non-private problem, the largest possible~$\norm{\tilde{x}}_2$ is bounded by the largest feasible~$\norm{x}_2$. We denote this as~$\bar{x}$ and say
% \begin{equation}
%     \bar{x} \in \arg\max_{x \in \mathcal{S}} \|x\|_2,
% \end{equation}
% where
% \begin{equation}
%     \mathcal{S} = \{x \in \mathbb{R}^n \mid Ax \leq b\}. 
% \end{equation}
% We then bound the expectation as
% \begin{equation}\label{eq:sub_z_in_here}
%     \Expectation{\norm{A^{\dagger}Z\tilde{x}^*}_2}\leq \norm{A^{\dagger}}_F\norm{\bar{x}}_2\Expectation{\norm{Z}_F}.
% \end{equation}
Using the definition of the Frobenius norm and Jensen's inequality, we use the preceding inequality to find the bound
%\begin{multline}
    $H_{2,2}(A)\norm{\bar{x}}_{2}\Expectation{\norm{Z}}\leq 
    %H_{2,2}(A)\norm{\bar{x}}_{2}\sqrt{\Expectation{\trace{Z^TZ}}} \\= 
    H_{2,2}(A)\norm{\bar{x}}_{2}\sqrt{\trace{\Expectation{Z^TZ}}}$.
%\end{multline}
Now we compute the diagonal entries of~$\Expectation{Z^TZ}$.
We break~$Z$ down into two cases: the case where there exists an~$i$ and~$j$ such that~$a_{i,j}+2s_i\geq \sup_{\seta} a_{i,j}$, and the case where~$a_{i,j}+2s_i< \sup_{\seta} a_{i,j}$ for all~$i$ and~$j$. 
Starting with the case~$a_{i,j}+2s_i< \sup_{\seta} a_{i,j}$ for all~$i$ and~$j$, we have
\begin{multline}
    (Z^TZ)_{1,1} = (s_1+\eta_{1,1})(s_1+\eta_{1,1})+(s_2+\eta_{2,1})(s_2+\eta_{2,1})\\+\cdots+(s_m+\eta_{m,1})(s_m+\eta_{m,1})
=s_1^2+2s_1\eta_{1,1}+\eta_{1,1}^2\\+s_2^2+2s_2\eta_{2,1}+\eta_{2,1}^2+\cdots+s_m^2+2s_m\eta_{m,1}+\eta_{m,1}^2.
\end{multline}
Each~$\eta^j_{i}$ has~$0$ mean and~$2\left(\frac{k}{\epsilon}\right)^2$ second moment. Thus, 
%\begin{equation}
    $\Expectation{(Z^TZ)_{1,1}} 
    %s_1^2 +2\left(\frac{k}{\epsilon}\right)^2 +s_2^2+2\left(\frac{k}{\epsilon}\right)^2+\ldots\\+s_m^2+2\left(\frac{k}{\epsilon}\right)^2
    =2n^0_1\left(\frac{k}{\epsilon}\right)^2+\sum_{i=1}^{n^0_1} s_i^2$.
%\end{equation}
This pattern holds for each diagonal entry, so we have
%\begin{equation}
$\Expectation{(Z^TZ)_{i,i}}=2n^0_i\left(\frac{k}{\epsilon}\right)^2+\sum_{i=1}^{n^0_i} s_i^2.$
%\end{equation}
The trace is then the sum of these identical diagonal entries, and because $Z^TZ\in\mathbb{R}^{m\times m}$, there are~$m$ diagonal entries, and thus
we may write the equality 
%\begin{equation}
    $\sqrt{\trace{\Expectation{Z^TZ}}} = \sqrt{\sum_{j=1}^{m}\left(2m\left(\frac{k}{\epsilon}\right)^2 n^0_j+n^0_j s_j^2\right)}$,
%2mn\left(\frac{k}{\epsilon}\right)^2+n\sum_{i=1}^m \left(\frac{k}{\epsilon}\log\left(\frac{n_i^0(e^{\epsilon}-1)}{\delta}+1\right)\right)^2}
%\end{equation}
where~$s_i$ is the support of the truncated Laplace mechanism for constraint~$i$. In the event that there exists an~$i$ and~$j$ such that~$a_{i,j}+2s_i\geq \sup_{\seta} a_{i,j}$, define~$\bar{A} = [\sup_{\seta} a_{i,j}]_{i\in[m], j\in[n]}$, that is, the matrix such that every entry in~$A$ takes the maximum possible value allowed by the set~$\seta$. In this case, ${Z_{ij} \leq (\bar{A}-A)_{ij}}$, and we have
%\begin{equation}
    $H_{2,2}(A)\norm{\bar{x}}_{2}\sqrt{\trace{\Expectation{Z^TZ}}} \leq H_{2,2}(A)\norm{\bar{x}}_{2}\norm{A-\bar{A}}_F.$
%\end{equation}
Defining~$\xi$ in~\eqref{eq:xi}, we substitute~$\xi$ into~\eqref{eq:sub_z_here} to find
%\begin{equation}
    $H_{2,2}(A)\norm{\bar{x}}_{2}\Expectation{\norm{Z}_F} \leq H_{2,2}(A)\norm{\bar{x}}_{2}\xi,$
%\end{equation}
which we substitute into~\eqref{eq:sub_norm_here} to obtain
%\begin{equation}
    %$\Expectation{|g(x^*)- g(\tilde{x}^*)|} \leq LH_{2,2}(A)\norm{\bar{x}}_{2}\xi,$
%\end{equation}
%which completes the proof.
the result.

\samepage

\end{document}